\documentclass{amsart}

\usepackage{amsmath,amssymb,amsthm,a4wide,bbm}
\usepackage{graphicx,tikz}

\newtheorem*{thm}{Theorem}
\newtheorem*{proposition}{Proposition}

\theoremstyle{definition}

\theoremstyle{remark}

\begin{document}

\title[]{ Generalized Designs on Graphs: \\ Sampling, Spectra, Symmetries}

\author[]{Stefan Steinerberger}
\address{Department of Mathematics, Yale University, New Haven, CT 06511, USA}
\email{stefan.steinerberger@yale.edu}

\begin{abstract} Spherical Designs are finite sets of points on the sphere $\mathbb{S}^{d}$ with the property
that the average of certain (low-degree) polynomials in these points coincides with the global average of the polynomial on $\mathbb{S}^{d}$.
They are evenly
distributed and often exhibit a great degree of regularity and symmetry. We point
out that a spectral definition of spherical designs transfers to finite graphs -- these 'graphical designs' are
subsets of vertices that are evenly spaced and capture the symmetries of the underlying graph (should they exist). Our main result states that good graphical designs either consist
of many vertices or their neighborhoods have exponential volume growth. We show several examples, describe ways to find them and discuss problems.
\end{abstract}

\maketitle

\section{Introduction}
\subsection{Spherical Designs.} Suppose $\left\{x_1, \dots, x_n\right\} \subset \mathbb{S}^2$ has the property that, for some weights $a_k$, 
$$ \frac{1}{|\mathbb{S}^2|} \int_{\mathbb{S}^2}{ f(x)dx} =  \sum_{k=1}^{n}{a_k f(x_k)} $$
for all polynomials $f$ up to a certain degree: depending on $n$, how large can the the degree of the polynomials be? A counting argument suggests that every one of the $n$ points has 2 coordinates and 1 weight: the right-hand side therefore has $3n$ degrees of freedom and one could hope to be able to integrate at least the first $\sim 3n$ low-degree polynomials exactly. This intuition was formulated by McLaren \cite{mclaren} in 1963 (see Ahrens \& Beylkin \cite{ahrens} for recent numerical experiments). 

\begin{figure}[h!]
\begin{tikzpicture}[scale=0.7]
  \tikzstyle{every node}=[circle,inner sep=0pt,minimum size=0.5cm]
    \foreach \y[count=\a] in {10,9,4}
      {\pgfmathtruncatemacro{\kn}{120*\a-90}
       \node at (\kn:3) (b\a) {\small \y};}
    \foreach \y[count=\a] in {8,7,2}
      {\pgfmathtruncatemacro{\kn}{120*\a-90}
       \node at (\kn:2.2) (d\a) {\small \y};}
    \foreach \y[count=\a] in {1,5,6}
      {\pgfmathtruncatemacro{\jn}{120*\a-30}
       \node at (\jn:1.5) (a\a) {\small \y};}
    \foreach \y[count=\a] in {3,11,12}
      {\pgfmathtruncatemacro{\jn}{120*\a-30}
       \node at (\jn:3) (c\a) {\small \y};}
  \draw[dashed] (a1)--(a2)--(a3)--(a1);
  \draw[ultra thick] (d1)--(d2)--(d3)--(d1);
  \foreach \a in {1,2,3}
   {\draw[dashed] (a\a)--(c\a);
   \draw[ultra thick] (d\a)--(b\a);}
   \draw[ultra thick] (c1)--(b1)--(c3)--(b3)--(c2)--(b2)--(c1);
   \draw[ultra thick] (c1)--(d1)--(c3)--(d3)--(c2)--(d2)--(c1);
   \draw[dashed] (b1)--(a1)--(b2)--(a2)--(b3)--(a3)--(b1);
\end{tikzpicture}
\caption{The Icosahedron has a great degree of symmetry. It integrates all polynomials on $\mathbb{S}^2$ up to degree 5 exactly, this vector space is 36-dimensional.}
\end{figure}
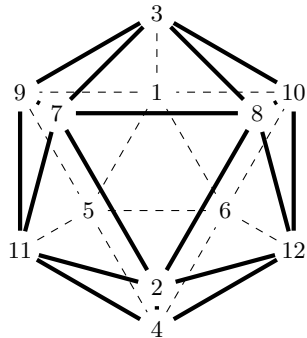

McLaren also described 72 points on $\mathbb{S}^2$ integrating $225 > 216 = 3 \cdot 72$ polynomials exactly:
\begin{quote}
Symmetries enable quadrature rules to exceed the linear algebra heuristic. 
\end{quote}

Symmetries of the underlying geometry translate into symmetries of the underlying smooth functions which can then be exploited to break the restriction imposed by linear algebra. However, this is a rare phenomenon and few examples are known.

\subsection{Graphs.} The purpose of this paper is to point out that this idea, demanding exact integration of relatively smooth objects to define well-distributed subsets of the geometry that reflect the underlying symmetries, can be easily and explicitly studied on finite Graphs $G=(V,E)$. Furthermore, all computational aspects boil down to linear algebra which is fairly accessible (possibly even easier than handling sets of points on $\mathbb{S}^2$ in the continuous case).
We will now formally define the problem. Polynomials on $\mathbb{S}^{d}$ are merely the eigenfunctions of the Laplace operator $-\Delta_{\mathbb{S}^{d}}$. This suggests making use of the eigenvectors of the discrete Laplace matrix 
$$ L = A  D^{-1}- I_{n \times n} \qquad \mbox{or, equivalenty,} \qquad (Lf)(u) =\sum_{v \sim_{E} u}{ \left( \frac{f(v)}{\mbox{deg}(v)}  - \frac{f(u)}{\mbox{deg}(u)}\right)}$$
where $\mbox{Id}_{n \times n}$ is the Identity matrix, $A$ is the adjacency matrix, $D$ is the diagonal degree matrix ($D_{ii}$ being the degree of the vertex $i$) and the summation runs over all neighbors $v$ of $u$.
 There are other  Laplacians, most notably the normalized Laplacian $L = \mbox{Id}_{n \times n} - D^{-1/2}AD^{-1/2}$ (see \cite{chung}). These are certainly also of great interest; we emphasize that most Graphs arising as explicit examples in this paper are regular: for regular Graphs both notions of a Laplacian agree.

\begin{center}
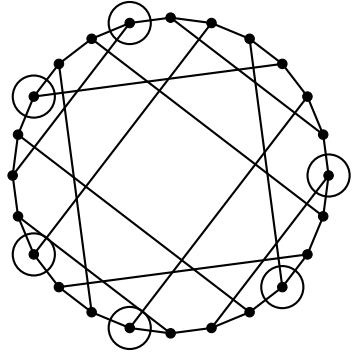
\begin{figure}[h!]
  \begin{tikzpicture}[scale=0.7]
\foreach \a in {1,2,...,24}{
\filldraw (\a*360/24: 3cm) circle (0.09cm);
};
\foreach \a in {1,2,...,24}{
\draw [thick] (\a*360/24: 3cm) --  (\a*360/24 + 360/24: 3cm);
};
\draw [thick] (1*360/24: 3cm) -- (6*360/24: 3cm);
\draw [thick] (2*360/24: 3cm) -- (17*360/24: 3cm);
\draw [thick] (3*360/24: 3cm) -- (10*360/24: 3cm);
\draw [thick] (4*360/24: 3cm) -- (21*360/24: 3cm);
\draw [thick] (5*360/24: 3cm) -- (14*360/24: 3cm);
\draw [thick] (7*360/24: 3cm) -- (12*360/24: 3cm);
\draw [thick] (8*360/24: 3cm) -- (23*360/24: 3cm);
\draw [thick] (9*360/24: 3cm) -- (16*360/24: 3cm);
\draw [thick] (11*360/24: 3cm) -- (20*360/24: 3cm);
\draw [thick] (13*360/24: 3cm) -- (18*360/24: 3cm);
\draw [thick] (15*360/24: 3cm) -- (22*360/24: 3cm);
\draw [thick] (19*360/24: 3cm) -- (24*360/24: 3cm);
\draw [thick] (7*360/24: 3cm) circle (0.4cm);
\draw [thick] (10*360/24: 3cm) circle (0.4cm);
\draw[thick] (14*360/24: 3cm) circle (0.4cm);
\draw [thick] (17*360/24: 3cm) circle (0.4cm);
\draw[thick] (21*360/24: 3cm) circle (0.4cm);
\draw [thick] (24*360/24: 3cm) circle (0.4cm);
   \end{tikzpicture}
\caption{The Nauru Graph on 24 vertices: a subset of 6 vertices integrates the first 19 eigenfunctions exactly. We observe that every other vertex is \textit{exactly} distance 1 away from exactly one of the 6 vertices.}
\end{figure}
\end{center}
\vspace{-10pt}
The operator $AD^{-1}$ corresponds to an averaging operator implying $\sigma(AD^{-1}) \subset [-1,1]$ and
$$ \sigma(L) \subset [-2,0].$$
Eigenvalues of $L$ with value close to $-1$ correspond to eigenfunctions that decay quickly under diffusion and should therefore be understood as the high-frequency objects. This motivates an ordering of eigenvalues from
low frequency to high frequency
$$ | \lambda_1 + 1| \geq |\lambda_2 +1| \geq |\lambda_3+1| \geq \dots \geq |\lambda_n+1|$$
We denote the associated eigenvectors by $\phi_1, \phi_2, \dots, \phi_n$ -- we will interpret these as functions $\phi_j:V \rightarrow \mathbb{R}$ and
write $\phi_j(v)$ to denote the value of the $j-$th eigenfunction at the vertex $v \in V$.

\subsection{The Problem.} The problem can now be formally posed as follows: among all subsets $W \subset V$ of a certain size $|W|$, is there one integrating many eigenfunctions exactly?\\

\begin{quote}
\textbf{Problem} (Graphical Design)\textbf{.} Let $G= (V,E)$ be a finite, simple, connected graph. Suppose there exists a subset $W \subset V$ with weights $a_w$ such that
$$ \forall ~1 \leq k \leq K: \qquad \sum_{w \in W}{ a_w \phi_k(w)} = \frac{1}{|V|} \sum_{v \in V}{  \phi_k(v)}$$
How big can $K$ be (depending on $|W|$)? How does this depend on
$G$? How would one find sets $W$ having $K$ large?\\
\end{quote}
We observe that there is some ambiguity when eigenspaces have a large multiplicity -- this will not be important throughout the paper (all numerical results are with respect to \textit{some} admissible ordering
of the eigenvectors). We will also use the terms eigenvector and eigenfunction interchangeably.
Linear Algebra suggests that for a generic set $W$ of vertices one would expect the existence of weights $a_w$ such that the first $\sim |W|$ eigenfunctions are integrated exactly.
All numerical examples in this paper have constant weights $a_w \equiv 1/|W|$. Generically, we do not expect any subset $W \subsetneq V$ with equal weights 
to integrate even the first nontrivial eigenfunction exactly: nontrivial graphical designs with equal weights are only possible in the presence of additional structure.\\

There is a natural motivation for this question that carries over from the continuous setting: if we are given a graph $G=(V,E)$ with many vertices $|V| \gg 1$ and a function
$f: V \rightarrow \mathbb{R}$ that is 'smooth' with respect to the geometry of $G$, then graphical designs are a natural place to sample to get a decent approximation for
the average value of $f$: they cancel low-frequency oscillations to the best of their ability. In the classical setting of $\mathbb{S}^d$, this idea can be found at least as early as 1962 in a paper of Sobolev \cite{sobolev} where it is discussed for the sphere and spherical harmonics. Lebedev et al. \cite{leb1, leb5} gave explicit constructions on $\mathbb{S}^2$ and this idea is now sometimes known as Sobolev-Lebedev quadrature. We can think of graphical designs as 'spherical designs on graphs' or 'Sobolev-Lebedev quadrature rules on graphs'.

\begin{center}
\begin{figure}[h!]
  \begin{tikzpicture}[scale=0.7]
\foreach \a in {1,2,...,9}{
\filldraw (\a*360/9: 3cm) circle (0.08cm);
};
\foreach \a in {1,2,...,9}{
\draw [thick] (\a*360/9: 3cm) --  (\a*360/9 + 360/9: 3cm);
};
\filldraw [thick] (0.8, -0.4) circle (0.08cm);
\filldraw [thick] (-0.8, -0.4) circle (0.08cm);
\filldraw [thick] (0, 1) circle (0.08cm);
\draw [thick] (0.8, -0.4) --  (8*360/9: 3cm);
\draw [thick] (0.8, -0.4) -- (-0.8, -0.4) -- (0,1) -- (0.8, -0.4);
\draw [thick] (-0.8, -0.4) --  (5*360/9: 3cm);
\draw [thick] (0,1) --  (2*360/9: 3cm);
\draw [thick] (1*360/9: 3cm) --  (3*360/9: 3cm);
\draw [thick] (4*360/9: 3cm) --  (6*360/9: 3cm);
\draw [thick] (7*360/9: 3cm) --  (9*360/9: 3cm);
\draw [ thick]  (1*360/9: 3cm) circle (0.4cm);
\draw [ thick]  (9*360/9: 3cm) circle (0.4cm);
\draw [ thick]  (-0.8, -0.4) circle (0.4cm);
\draw [ thick]  (5*360/9: 3cm) circle (0.4cm);
   \end{tikzpicture}
\caption{The Truncated Tetrahedral Graph on 12 vertices: a subset of 4 vertices integrates the first 11 eigenfunctions exactly. Every other vertex is exactly distance 1 away from exactly one of the 4 vertices.}
\end{figure}
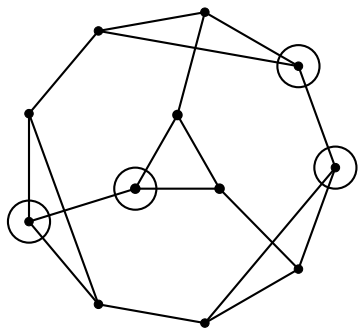
\end{center}

\vspace{-20pt}

\subsection{Related Literature.} The systematic study of spherical $t-$designs was started by Delsarte, Goethals \& Seidel \cite{delsarte} (based on some earlier ideas of Delsarte \cite{del} on $t-$designs in 
$Q-$polynomial association schemes).
 It is impossible to summarize the field, we mention
the seminal papers by Bondarenko, Radchenko \& Viazovska \cite{bond} and Yudin \cite{yudin} and refer to a recent survey of Brauchart \& Grabner \cite{brau}. A Lemma
of Montgomery \cite{mont} (see also \cite{bilyk, neu, steinriesz}) may be understood as the study of the analogue of spherical designs on $\mathbb{T}^d$ where polynomials are replaced by trigonometric polynomials  -- this result does not seem to be very well known in this community since the relevant statement appears as a Lemma and is used for a very different purpose. 
To the best of our knowledge,
the first upper bound for weighted spherical designs on Riemannian manifolds is due to the author \cite{stein}; the eigenfunctions of the Laplacian
on $\mathbb{S}^2$ are polynomials, the classical setting is therefore included as a special case. We briefly remark that there is also the study of combinatorial designs (see e.g. \cite{ass}): these are families of a subsets with highly structured intersection patterns. Seidel \cite{seidel}
refers to spherical designs as their 'Euclidean counterpart'. We also emphasize that the notion of $t-$design in Delsarte's seminal paper \cite{del} has rather striking implications for distance-regular graphs and strong ties to combinatorial designs: Delsarte himself proved that $t-$designs in Johnson Graphs and the Hamming graphs are the combinatorial block $t-$designs and the orthogonal arrays of strength $t$, respectively; we refer to the books of Bannai \& Ito \cite{ban} and Brouwer, Cohen \& Neumaier \cite{bro2} for more details.

\section{The Result}

We now state the main result: if there exists a set $W$ that integrates eigenfunctions up to a certain eigenvalue exactly, then either 
$|W|$ is large or there is exponential growth of neighborhoods.

\begin{thm}[]  Let $V=(G,E)$ be a finite simple graph such that $L =A D^{-1} - \emph{Id}_{n \times n}$ has an orthogonal set of eigenvectors.
Let $W \subset V$ be a subset equipped with positive weights normalized such that
$$ \sum_{w \in W}{ a_w \phi_k(w)} = \frac{1}{|V|} \sum_{v \in V}{  \phi_k(v)}$$
for all eigenvectors $\phi_k$ whose eigenvalue satisfies $| \lambda_k + 1| \geq \lambda$. Then, for every $k \in \mathbb{N}$,
$$  \# \left\{ x \in V: d(x,W) \leq k\right\} \geq  \frac{1}{2} \min \left\{  \frac{1}{\lambda^{2k}}, |V| \right\}.$$
Moreover, if all weights are identical, $|a_w| = 1/|W|$, then we have 
$$  \# \left\{ x \in V: d(x,W) \leq k\right\} \geq  \frac{1}{2} \min \left\{  \frac{|W|}{\lambda^{2k}}, |V| \right\}.$$
\end{thm}

The theorem can be interpreted in a number of ways. The most obvious one is that if $|W|$ is small, then either the eigenfunctions that are being integrated exactly
are not deep inside the spectrum or the $k-$neighborhoods of $W$ undergo exponential growth all the way until they contain half of all vertices.
This can be seen rather drastically in various explicit examples throughout the paper.

\begin{center}
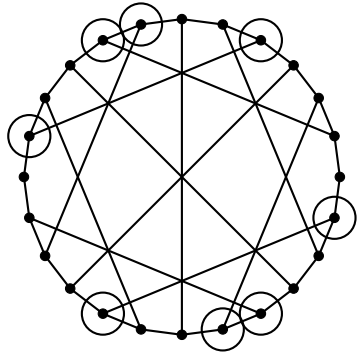
\begin{figure}[h!]
  \begin{tikzpicture}[scale=0.7]
\foreach \a in {1,2,...,24}{
\filldraw (\a*360/24: 3cm) circle (0.09cm);
};
\foreach \a in {1,2,...,24}{
\draw [thick] (\a*360/24: 3cm) --  (\a*360/24 + 360/24: 3cm);
};
\draw [thick] (1*360/24: 3cm) -- (8*360/24: 3cm);
\draw [thick] (2*360/24: 3cm) -- (19*360/24: 3cm);
\draw [thick] (3*360/24: 3cm) -- (15*360/24: 3cm);
\draw [thick] (4*360/24: 3cm) -- (11*360/24: 3cm);
\draw [thick] (5*360/24: 3cm) -- (22*360/24: 3cm);
\draw [thick] (6*360/24: 3cm) -- (18*360/24: 3cm);
\draw [thick] (7*360/24: 3cm) -- (14*360/24: 3cm);
\draw [thick] (9*360/24: 3cm) -- (21*360/24: 3cm);
\draw [thick] (10*360/24: 3cm) -- (17*360/24: 3cm);
\draw [thick] (13*360/24: 3cm) -- (20*360/24: 3cm);
\draw [thick] (16*360/24: 3cm) -- (23*360/24: 3cm);
\draw [thick] (4*360/24: 3cm) circle (0.4cm);
\draw [thick] (7*360/24: 3cm) circle (0.4cm);
\draw [thick] (8*360/24: 3cm) circle (0.4cm);
\draw[thick] (11*360/24: 3cm) circle (0.4cm);
\draw [thick] (16*360/24: 3cm) circle (0.4cm);
\draw[thick] (19*360/24: 3cm) circle (0.4cm);
\draw[thick] (20*360/24: 3cm) circle (0.4cm);
\draw [thick] (23*360/24: 3cm) circle (0.4cm);
   \end{tikzpicture}
\caption{The McGee Graph on 24 vertices: a subset of 8 vertices integrates the first 21 eigenfunctions exactly. Every other vertex is exactly distance 1 away from exactly one element of this subset.}
\end{figure}
\end{center}

\vspace{-15pt}

There are other possible applications: if the Graph is $d-$regular, then
$$   \# \left\{ x \in V: d(x,W) \leq k\right\} \leq |W| ( d + d^2 + \dots + d^k) \lesssim_d |W| \cdot d^k$$
and we see that subsets $W \subset V$ with equal weight can never be exact on all eigenvectors whose eigenvalue is  $|\lambda + 1| \lesssim d^{-1/2}$.
Another implication is as follows: if we are able to integrate all eigenvectors with eigenvalues $|\lambda_k + 1| \geq 0.99$ exactly, then either $|W| \sim |V|$ or, if $|W| \ll |V|$, then the Graph has
most of its vertices at distance $\sim \log{|V|}$ from each other. The proof of the Theorem does not seem
to have any obvious bottleneck; the result could be close to optimal. The proof immediately extends to Graphs with weighted
edges as long as the Laplacian is defined in such a way that the associated diffusion preserves the mean value of a function.

\section{Finding Graphical Designs} There is a large number of theoretical questions that would greatly benefit
form having many nice examples and the ability to quickly see whether a Graph supports a graphical
design. The
trivial algorithm is to simply go through all possible subsets of a certain cardinality and test them. However, this scales badly and
already Graphs with, say, $|V| \sim 25$ already require considerable time. Nonetheless, brute force search is a viable option when
studying 'small' Graphs. When it comes to proving the non-existence of a graphical design of a certain size and quality, we do
not know of any other method. However, the construction of explicit examples can be accelerated.\\

\textbf{A Simple Algorithm.} Good qualities as a graphical design correspond to rapid growth of neighborhoods.
This suggests that we should pick $W$ to be as spread out as possible. A fairly simple but somewhat effective algorithm is as follows.

\begin{quote} 
\textbf{Algorithm.} Pick $k$ random elements $\left\{v_1, \dots, v_k\right\}$ from $V$. Compute
$$ \mbox{total pairwise distance} = \sum_{i,j=1}^{k}{d(v_i, v_j)}.$$
Go through all elements and see whether replacing a vertex by one of its neighbors increases the sum. If so, then flip a coin
and either replace the vertex by this new vertex or not. Repeat this procedure on the new configuration and check only subsets arising in this process
for their quality as a graphical design. 
\end{quote}
The coin flip is supposed to stabilize possibly arising oscillations. The choice of the functional, i.e. the total sum of the distances,
is somewhat natural but certainly not canonical. We have found that this simple algorithm performs quite well even on somewhat 
large graphs, see Table 1. However, it is of course difficult to say whether the results obtained are even close to optimal or whether there
are much better graphical designs on these particular Graphs. 
\begin{center}
\begin{table}[h!] \label{tab}
\begin{tabular}{l  c  c  c r}
Graph &  $|V|$   &  $|E|$ & $|W|$& $\#$ eigenfunctions\\
\hline
24-Cell Graph  & 24  & 96 & 3  &  8\\
Icosidodecahedral Graph  &30  & 60 & 6  &  24\\
Cayley Graph (30,1)  & 30  & 60 & 6  &  19\\
Gewirtz Graph  & 56 & 280 & 1 & 19 \\
Gosset Graph  & 56  & 756 & 4  &  29\\
\vspace{0pt}
\end{tabular}
\caption{Graphs with $|V|$ vertices and $|E|$ edges, there exists a set of vertices $W \subset V$ such that $\#$ eigenfunctions are integrated exactly.}
\end{table}
\end{center}
\vspace{-20pt}
This approach is close in spirit to the idea that minimizing energy configuration of a finite number of particles on manifolds should be good integration
points. This intuition is rarely made precise; the author showed in \cite{steinriesz} that the minimal energy configuration of $N$ particles on the torus $\mathbb{T}^d$ under Gaussian interaction 
is an optimal set of quadratures in the space of low-degree trigonometric polynomials -- we will obtain a similar result here and discuss in \S 5.1, as a natural byproduct of the proof, an 'almost-characterization' of good graphical designs
that can be used to refine the notion of distance in the algorithm above.

\begin{center}
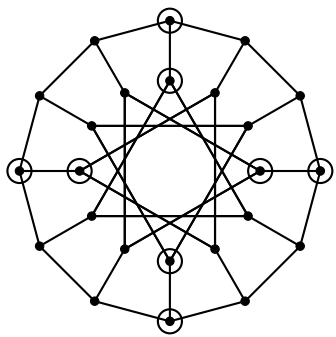
\begin{figure}[h!]
  \begin{tikzpicture}[scale=0.4]
\foreach \a in {1,2,...,12}{
\filldraw (\a*360/12: 3cm) circle (0.14cm);
};
\foreach \a in {1,2,...,12}{
\draw [thick] (\a*360/12: 3cm) --  (\a*360/12 + 4*360/12: 3cm);
\draw [thick] (\a*360/12: 3cm) --  (\a*360/12 + 8*360/12: 3cm);
\draw [thick] (\a*360/12: 3cm) --  (\a*360/12: 5cm);
};
\foreach \a in {1,2,...,12}{
\filldraw (\a*360/12: 5cm) circle (0.14cm);
};
\foreach \a in {1,2,...,12}{
\draw [thick] (\a*360/12: 5cm) --  (\a*360/12 + 360/12: 5cm);
};
\draw [thick] (3*360/12: 3cm) circle (0.4cm);
\draw [thick] (6*360/12: 3cm) circle (0.4cm);
\draw [thick] (9*360/12: 3cm) circle (0.4cm);
\draw [thick] (12*360/12: 3cm) circle (0.4cm);
\draw [thick] (15*360/12: 5cm) circle (0.4cm);
\draw [thick] (18*360/12: 5cm) circle (0.4cm);
\draw [thick] (21*360/12: 5cm) circle (0.4cm);
\draw [thick] (24*360/12: 5cm) circle (0.4cm);
   \end{tikzpicture}
\caption{Generalized Petersen Graph (12,4) on 24 vertices: a subset of 8 vertices integrates the first 22 eigenfunctions exactly: its neighborhood is all of $V$.}
\end{figure}
\end{center}

We emphasize that the efficiency of the simple algorithm may be simply due to the fact that distance-maximizing configurations tend to exploit symmetries and have a tendency
to end up in highly symmetric arrangements. However, it certainly succeeds in finding graphical designs of small cardinality in graphs with high degrees of symmetry. For an algorithm with real
theoretical justification, we refer to \S 5.1.

\section{Proof}
\subsection{Proof of the Theorem.}
\begin{proof} We assume that $G=(V,E)$ is given, abbreviate $n = |V|$ and assume that the operator $L = AD^{-1} - \mbox{Id}_{n \times n}$ has a set of orthogonal eigenvectors whose eigenvalues are indexed as
$$  1 = |\lambda_1 + 1| \geq |\lambda_2 + 1| \geq \dots \geq |\lambda_n + 1|.$$

This is merely the correct ordering in terms of absolute value of their eigenvalue w.r.t. the matrix $AD^{-1}$. 
Assume that $W \subset V$ is a set of vertices equipped with positive weights $a_w \geq 0$ such that
$$\sum_{w \in W}{a_w \phi_k(w)} = \frac{1}{|V|} \sum_{v \in V}{ \phi_k(v)}$$
for all eigenfunctions $\phi_k$ whose eigenvalue satisfies $|\lambda_k + 1| \geq \lambda > 0$. The first eigenfunction is constant, this implies the normalization
$$ \sum_{w \in W}{a_w} = 1.$$
We observe that $\mbox{Id}_{n \times n} + L = AD^{-1}$ is the transition probability for the random walk on the Graph. We also observe that it preserves the average of
a function since

$$ \sum_{v \in V}{\left[ (AD^{-1})f\right](v)} = \sum_{v \in V}{  \sum_{(w,v) \in E}{ \frac{f(w)}{\mbox{deg}(w)}}} = \sum_{v \in V}{f(v)}.$$

\begin{center}
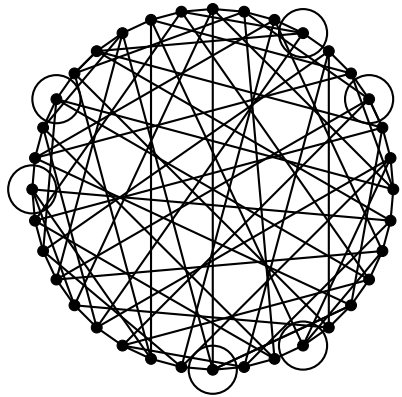
\begin{figure}[h!]
  \begin{tikzpicture}[scale=0.8]
\foreach \a in {1,2,...,36}{
\filldraw (\a*360/36: 3cm) circle (0.09cm);
};
\foreach \a in {1,2,...,36}{
\draw [thick] (\a*360/36: 3cm) --  (\a*360/36 + 360/36: 3cm);
};
\draw [thick] (1*360/36: 3cm) -- (23*360/36: 3cm);
\draw [thick] (1*360/36: 3cm) -- (26*360/36: 3cm);
\draw [thick] (1*360/36: 3cm) -- (30*360/36: 3cm);
\draw [thick] (2*360/36: 3cm) -- (3*360/36: 3cm);
\draw [thick] (2*360/36: 3cm) -- (32*360/36: 3cm);
\draw [thick] (2*360/36: 3cm) -- (13*360/36: 3cm);
\draw [thick] (2*360/36: 3cm) -- (19*360/36: 3cm);
\draw [thick] (3*360/36: 3cm) -- (8*360/36: 3cm);
\draw [thick] (3*360/36: 3cm) -- (21*360/36: 3cm);
\draw [thick] (3*360/36: 3cm) -- (24*360/36: 3cm);
\draw [thick] (4*360/36: 3cm) -- (11*360/36: 3cm);
\draw [thick] (4*360/36: 3cm) -- (17*360/36: 3cm);
\draw [thick] (4*360/36: 3cm) -- (36*360/36: 3cm);
\draw [thick] (5*360/36: 3cm) -- (20*360/36: 3cm);
\draw [thick] (5*360/36: 3cm) -- (28*360/36: 3cm);
\draw [thick] (5*360/36: 3cm) -- (31*360/36: 3cm);
\draw [thick] (6*360/36: 3cm) -- (10*360/36: 3cm);
\draw [thick] (6*360/36: 3cm) -- (13*360/36: 3cm);
\draw [thick] (6*360/36: 3cm) -- (23*360/36: 3cm);
\draw [thick] (7*360/36: 3cm) -- (17*360/36: 3cm);
\draw [thick] (7*360/36: 3cm) -- (26*360/36: 3cm);
\draw [thick] (7*360/36: 3cm) -- (34*360/36: 3cm);
\draw [thick] (8*360/36: 3cm) -- (14*360/36: 3cm);
\draw [thick] (8*360/36: 3cm) -- (29*360/36: 3cm);
\draw [thick] (9*360/36: 3cm) -- (22*360/36: 3cm);
\draw [thick] (9*360/36: 3cm) -- (27*360/36: 3cm);
\draw [thick] (9*360/36: 3cm) -- (36*360/36: 3cm);
\draw [thick] (10*360/36: 3cm) -- (19*360/36: 3cm);
\draw [thick] (10*360/36: 3cm) -- (33*360/36: 3cm);
\draw [thick] (11*360/36: 3cm) -- (25*360/36: 3cm);
\draw [thick] (11*360/36: 3cm) -- (30*360/36: 3cm);
\draw [thick] (12*360/36: 3cm) -- (16*360/36: 3cm);
\draw [thick] (12*360/36: 3cm) -- (21*360/36: 3cm);
\draw [thick] (12*360/36: 3cm) -- (27*360/36: 3cm);
\draw [thick] (13*360/36: 3cm) -- (14*360/36: 3cm);
\draw [thick] (13*360/36: 3cm) -- (35*360/36: 3cm);
\draw [thick] (14*360/36: 3cm) -- (25*360/36: 3cm);
\draw [thick] (14*360/36: 3cm) -- (31*360/36: 3cm);
\draw [thick] (15*360/36: 3cm) -- (20*360/36: 3cm);
\draw [thick] (15*360/36: 3cm) -- (33*360/36: 3cm);
\draw [thick] (15*360/36: 3cm) -- (36*360/36: 3cm);
\draw [thick] (16*360/36: 3cm) -- (23*360/36: 3cm);
\draw [thick] (16*360/36: 3cm) -- (29*360/36: 3cm);
\draw [thick] (17*360/36: 3cm) -- (32*360/36: 3cm);
\draw [thick] (18*360/36: 3cm) -- (22*360/36: 3cm);
\draw [thick] (18*360/36: 3cm) -- (25*360/36: 3cm);
\draw [thick] (18*360/36: 3cm) -- (35*360/36: 3cm);
\draw [thick] (19*360/36: 3cm) -- (29*360/36: 3cm);
\draw [thick] (20*360/36: 3cm) -- (26*360/36: 3cm);
\draw [thick] (21*360/36: 3cm) -- (34*360/36: 3cm);
\draw [thick] (22*360/36: 3cm) -- (31*360/36: 3cm);
\draw [thick] (24*360/36: 3cm) -- (28*360/36: 3cm);
\draw [thick] (24*360/36: 3cm) -- (33*360/36: 3cm);
\draw [thick] (27*360/36: 3cm) -- (32*360/36: 3cm);
\draw [thick] (28*360/36: 3cm) -- (35*360/36: 3cm);
\draw [thick] (27*360/36: 3cm) -- (32*360/36: 3cm);
\draw [thick] (30*360/36: 3cm) -- (34*360/36: 3cm);
\draw [thick] (3*360/36: 3cm) circle (0.4cm);
\draw [thick] (6*360/36: 3cm) circle (0.4cm);
\draw [thick] (15*360/36: 3cm) circle (0.4cm);
\draw [thick] (18*360/36: 3cm) circle (0.4cm);
\draw [thick] (27*360/36: 3cm) circle (0.4cm);
\draw [thick] (30*360/36: 3cm) circle (0.4cm);
   \end{tikzpicture}
\caption{The Sylvester Graph on 36 vertices: 6 vertices integrate the first 26 eigenfunctions exactly. Every vertex is connected to one of the 6 vertices.}
\end{figure}
\end{center}

The main idea of the proof is an analysis of the evolution of the function 
$$ f =  - \frac{1}{n} + \sum_{w \in W}{a_w \delta_w},$$
under iterated applications of the operator $\mbox{Id}_{n \times n} + L$. Here, $\delta_w$ is the characteristic function on $w$, i.e.
$$ \delta_w(v) = \begin{cases} 1 \qquad &\mbox{if} ~v = w \\ 0 \qquad &\mbox{otherwise.} \end{cases}$$
We derive upper and lower bounds on the $L^2-$norm of $(\mbox{Id}_{n \times n} + L)^k f$, comparing these bounds to each other then yields the desired result.
We start with an expansion of $f$ into eigenvectors of $(\mbox{Id}_{n \times n} + L)$. We note that there is an eigenvector
having all constant entries and that all other eigenvectors are orthogonal by assumption: thus all but the first eigenvector have mean value 0. Moreover, $f$ also has mean value 0 and is thus
orthogonal to the first eigenvector. Since this weighted subset integrates the first few eigenfunctions exactly, we observe that
\begin{align*}
(\mbox{Id}_{n \times n} + L)^k f &= \sum_{i=1}^{n}{(\lambda_i +1)^k \left\langle f, \phi_i \right\rangle \phi_i }   =  \sum_{i=1}^{n}{(\lambda_i +1)^k \left\langle   - \frac{1}{n} + \sum_{w \in W}{a_w \delta_w}, \phi_i \right\rangle \phi_i } \\
&=   \sum_{i=2}^{n}{(\lambda_i +1)^k \left\langle  \sum_{w \in W}{a_w \delta_w}, \phi_i \right\rangle \phi_i }    = \sum_{|\lambda_i + 1| \leq \lambda}^{}{(\lambda_i +1)^k \left\langle  \sum_{w \in W}{a_w \delta_w}, \phi_i \right\rangle \phi_i }
\end{align*}
We can now use the Pythagorean theorem to conclude that
\begin{align*}
 \| (\mbox{Id}_{n \times n} + L)^k f\|^2_{L^2} &= \sum_{|\lambda_i + 1| \leq \lambda}^{}{|\lambda_i +1|^{2k}  \left|\left\langle  \sum_{w \in W}{a_w \delta_w}, \phi_i \right\rangle\right|^2 } \\
&\leq  \lambda^{2k} \sum_{i=1}^{n}{ \left|\left\langle  \sum_{w \in W}{a_w \delta_w}, \phi_i \right\rangle\right|^2 } = \lambda^{2k} \left\|  \sum_{w \in W}{a_w \delta_w} \right\|_{L^2}^2 \\
&= \lambda^{2k} \sum_{w \in W}{a_w^2}.
\end{align*}

\begin{center}
\begin{figure}[h!]
  \begin{tikzpicture}[scale=0.6]
\foreach \a in {1,2,...,18}{
\filldraw (\a*360/18: 3cm) circle (0.10cm);
};
\foreach \a in {1,2,...,18}{
\draw [thick] (\a*360/18: 3cm) --  (\a*360/18 + 360/18: 3cm);
};
\draw [thick] (360/18: 3cm) --  (6*360/18: 3cm);
\draw [thick] (2*360/18: 3cm) --  (9*360/18: 3cm);
\draw [thick] (3*360/18: 3cm) --  (14*360/18: 3cm);
\draw [thick] (4*360/18: 3cm) --  (11*360/18: 3cm);
\draw [thick] (5*360/18: 3cm) --  (16*360/18: 3cm);
\draw [thick] (7*360/18: 3cm) --  (12*360/18: 3cm);
\draw [thick] (8*360/18: 3cm) --  (15*360/18: 3cm);
\draw [thick] (10*360/18: 3cm) --  (17*360/18: 3cm);
\draw [thick] (13*360/18: 3cm) --  (18*360/18: 3cm);
\draw [ thick]  (1*360/18: 3cm)  circle  (0.3cm);
\draw [ thick]  (6*360/18: 3cm)  circle  (0.3cm);
\draw [ thick]  (10*360/18: 3cm)  circle  (0.3cm);
\draw [ thick]  (11*360/18: 3cm)  circle  (0.3cm);
\draw [ thick]  (14*360/18: 3cm)  circle  (0.3cm);
\draw [ thick]  (15*360/18: 3cm)  circle  (0.3cm);
   \end{tikzpicture}
\caption{Pappus Graph on 18 vertices: 6 vertices integrate the first 14 eigenfunctions. The $1-$neighborhood of the set is again all of $V$.}
\end{figure}
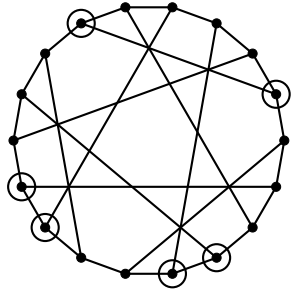
\end{center}

Clearly, since the weights are nonnegative and add up to 1, we get
$$  \| (\mbox{Id}_{n \times n} + L)^k f\|^2_{L^2}\leq \lambda^{2k}  \sum_{w \in W}{a_w^2} \leq \lambda^{2k}  \sum_{w \in W}{a_w} =  \lambda^{2k}.$$
If the weights are all equal, this bound improves to
$$  \| (\mbox{Id}_{n \times n} + L)^k f\|^2_{L^2} \leq \lambda^{2k}  \sum_{w \in W}{a_w^2}  = \lambda^{2k}  \sum_{w \in W}{\frac{1}{|W|^2}} =  \frac{\lambda^{2k}}{|W|}.$$
The remainder of the proof is devoted to deriving a lower bound for the expression.
We can use linearity and invariance of constants to conclude that
$$ (\mbox{Id}_{n \times n}  + L)^k f = -\frac{1}{n} + (\mbox{Id}_{n \times n}  + L)^k \sum_{w \in W}{a_w \delta_w}.$$
The operator $\mbox{Id}_{n \times n}  + L = AD^{-1}$ preserves the average value of a function. At the same time,
it acts as a discrete diffusion and in one time step can only transport mass to immediate neighbors. This, together with positivity of weights, implies that the support 
 $ (\mbox{Id}_{n \times n}  + L)^k f$ is given by 
$$  S_k = \left\{ v \in V:  \left[ (\mbox{Id}_{n \times n}  + L)^k \sum_{w \in W}{a_w \delta_w}\right] (v) > 0 \right\}  =  \left\{ x \in V: d(x,W) \leq k\right\}.$$
Squaring out implies
\begin{align*}
 \sum_{v \in V}{ \left( -\frac{1}{n} + (\mbox{Id}_{n \times n}  + L)^k \sum_{w \in W}{a_w \delta_w (v)}\right)^2} &=  \frac{1}{n} +  \left\|  (\mbox{Id}_{n \times n}  + L)^k \sum_{w \in W}{a_w \delta_w} \right\|_{L^2}^2 \\
&-   \sum_{v \in V}{  \left[ \frac{2}{n} (\mbox{Id}_{n \times n}  + L)^k \sum_{w \in W}{a_w \delta_w } \right](v)}
\end{align*}
However, the operator preserves integral averages and therefore
$$  \sum_{v \in V}{  \left[ \frac{2}{n} (\mbox{Id}_{n \times n}  + L)^k \sum_{w \in W}{a_w \delta_w } \right](v)} = \frac{2}{n}  \sum_{v \in V}{  \sum_{w \in W}{a_w \delta_w (v)}} = \frac{2}{n} \sum_{w \in W}{a_w} = \frac{2}{n}.$$

\vspace{-10pt}
\begin{center}
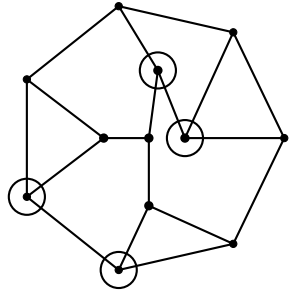
\begin{figure}[h!]
  \begin{tikzpicture}[scale=0.6]
\foreach \a in {1,2,...,7}{
\filldraw (\a*360/7: 3cm) circle (0.08cm);
};
\foreach \a in {1,2,...,7}{
\draw [thick] (\a*360/7: 3cm) --  (\a*360/7 + 360/7: 3cm);
};
\filldraw [thick] (0.8, 0) circle (0.08cm);
\draw [thick] (0.8, 0) --  (1*360/7: 3cm);
\draw [thick] (0.8, 0) --  (7*360/7: 3cm);
\draw [thick] (0.8, 0) --  (0.2, 1.5);
\filldraw [thick] (0.2, 1.5) circle (0.08cm);
\filldraw [thick] (0,0) circle (0.08cm);
\filldraw [ thick] (-1,0) circle (0.08cm);
\filldraw [ thick] (0,-1.5) circle (0.08cm);
\draw [ thick] (0.2, 1.5) --  (2*360/7: 3cm);
\draw [ thick] (0.2, 1.5) --  (0,0);
\draw [ thick] (-1, 0) --  (0,0);
\draw [ thick] (0, -1.5) --  (0,0);
\draw [ thick] (-1, 0) --  (3*360/7: 3cm);
\draw [ thick] (-1, 0) --  (4*360/7: 3cm);
\draw [ thick] (0,-1.5) --  (5*360/7: 3cm);
\draw [ thick] (0,-1.5) --  (6*360/7: 3cm);
\draw [ thick] (4*360/7: 3cm) circle (0.4cm);
\draw [ thick] (5*360/7: 3cm) circle (0.4cm);
\draw [ thick] (0.2, 1.5) circle (0.4cm);
\draw [ thick] (0.8, 0) circle (0.4cm);
   \end{tikzpicture}
\caption{The Frucht Graph on 12 vertices: a subset $W$ of 4 vertices integrates the first 11 eigenfunctions exactly. The set is again expanding optimality, the $1-$neighborhood of the set is all of $V$.}
\end{figure}
\end{center}
\vspace{-10pt}

This shows that
$$  \left\|    -\frac{1}{n} + (\mbox{Id}_{n \times n}  + L)^k \sum_{w \in W}{a_w \delta_w (v)} \right\|_{L^2}^2 = \left\|   (\mbox{Id}_{n \times n}  + L)^k \sum_{w \in W}{a_w \delta_w (v)}\right\|_{L^2}^2 - \frac{1}{n}.$$
We multiply the function with the characteristic function of its support and use the Cauchy-Schwarz inequality to argue that
\begin{align*}
 1 &=  \left\| \mathbbm{1}_{S_k} \cdot  (\mbox{Id}_{n \times n}  + L)^k \sum_{w \in W}{a_w \delta_w} \right\|^2_{L^1} \leq (\# S_k)   \left\|  (\mbox{Id}_{n \times n}  + L)^k \sum_{w \in W}{a_w \delta_w} \right\|_{L^2}^2.
\end{align*}

In the case where the $k-$neighborhood of $W$ contains more than $n/2$ vertices, there is nothing to show. If this is not the case, i.e. $\# S_k \leq n/2$, then this inequality shows that
$$  \left\|  (\mbox{Id}_{n \times n}  + L)^k \sum_{w \in W}{a_w \delta_w} \right\|_{L^2}^2 \geq \frac{2}{n}$$
and therefore
$$  \left\|    -\frac{1}{n} + (\mbox{Id}_{n \times n}  + L)^k \sum_{w \in W}{a_w \delta_w (v)} \right\|_{L^2}^2  \geq \frac{1}{2} \left\|  (\mbox{Id}_{n \times n}  + L)^k \sum_{w \in W}{a_w \delta_w} \right\|_{L^2}^2.$$

Altogether, this yields
$$   \frac{1}{2 } \frac{1}{\# \left\{ x \in V: d(x,W) \leq k\right\} } \leq   \left\|  (\mbox{Id}_{n \times n}  + L)^k f \right\|_{L^2}^2 \leq \lambda^{2k}$$
and we obtain the desired result. If all weights are $a_w = 1/|W|$, then we get an improvement by a factor of $|W|$ in the upper bound as outlined above. 
The factor $1/2$ is merely for convenience. The strongest result following from this argument is
$$ \frac{1}{\# \left\{ x \in V: d(x,W) \leq k\right\} } - \frac{1}{|V|} \leq \lambda^{2k} \sum_{w \in W}{a_w^2}.$$
\end{proof}

\section{Concluding Remarks.} 

\subsection{A Heat Kernel Packing Problem.} The proof has one immediate application that can be used in applications: it shows that in order for there to even be a chance of
having a very good graphical design, we require that
$$  \left\| - \frac{1}{n} +   (\mbox{Id}_{n \times n}  + L)^k \sum_{w \in W}{a_w \delta_w} \right\|_{L^2}^2 \qquad \mbox{is small.}$$
Conversely, \textit{if} the quantity is small, then the identity
$$ \left\| - \frac{1}{n} +  (\mbox{Id}_{n \times n}  + L)^k \sum_{w \in W}{a_w \delta_w} \right\|_{L^2}^2 = \sum_{i=2}^{n}{ |\lambda_i + 1|^{2k}  \left| \left\langle \sum_{w \in W}{a_w \delta_w} , \phi_i \right\rangle \right|^2}$$
implies that the right-hand side cannot have a lot of mass at low frequencies. Making the quantity small is thus not a guarantee of having a graphical design but at least guarantees an 'almost-graphical design': it does not necessarily integrate
eigenfunctions with $|\lambda_i + 1|$ large exactly but the error it makes cannot be large (depending on how small the term on the left-hand side is).

\begin{center}
\begin{figure}[h!]
  \begin{tikzpicture}[scale=0.8]
\foreach \a in {1,2,...,32}{
\filldraw (\a*360/32: 3cm) circle (0.09cm);
};
\foreach \a in {1,2,...,32}{
\draw [thick] (\a*360/32: 3cm) --  (\a*360/32 + 360/32: 3cm);
};
\draw [thick] (360/32: 3cm) --  (6*360/32: 3cm);
\draw [thick] (2*360/32: 3cm) --  (29*360/32: 3cm);
\draw [thick] (3*360/32: 3cm) --  (16*360/32: 3cm);
\draw [thick] (4*360/32: 3cm) --  (23*360/32: 3cm);
\draw [thick] (5*360/32: 3cm) --  (10*360/32: 3cm);
\draw [thick] (7*360/32: 3cm) --  (20*360/32: 3cm);
\draw [thick] (8*360/32: 3cm) --  (27*360/32: 3cm);
\draw [thick] (9*360/32: 3cm) --  (14*360/32: 3cm);
\draw [thick] (11*360/32: 3cm) --  (24*360/32: 3cm);
\draw [thick] (12*360/32: 3cm) --  (31*360/32: 3cm);
\draw [thick] (13*360/32: 3cm) --  (18*360/32: 3cm);
\draw [thick] (15*360/32: 3cm) --  (28*360/32: 3cm);
\draw [thick] (16*360/32: 3cm) --  (3*360/32: 3cm);
\draw [thick] (17*360/32: 3cm) --  (22*360/32: 3cm);
\draw [thick] (19*360/32: 3cm) --  (32*360/32: 3cm);
\draw [thick] (21*360/32: 3cm) --  (26*360/32: 3cm);
\draw [thick] (25*360/32: 3cm) --  (30*360/32: 3cm);
\draw [ thick]  (7*360/32: 3cm)  circle  (0.3cm);
\draw [ thick]  (10*360/32: 3cm)  circle  (0.3cm);
\draw [ thick]  (13*360/32: 3cm)  circle  (0.3cm);
\draw [ thick]  (16*360/32: 3cm)  circle  (0.3cm);
\draw [ thick]  (23*360/32: 3cm)  circle  (0.3cm);
\draw [ thick]  (26*360/32: 3cm)  circle  (0.3cm);
\draw [ thick]  (29*360/32: 3cm)  circle  (0.3cm);
\draw [ thick]  (32*360/32: 3cm)  circle  (0.3cm);

   \end{tikzpicture}
\caption{Dyck Graph on 32 vertices: 8 vertices integrate the first 16 eigenfunctions. Every vertex of $V$ is connected to exactly one of the 8 vertices.}
\end{figure}
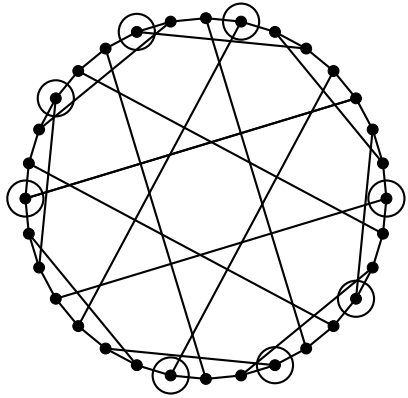
\end{center}

However, as observed in the proof,
$$  \left\| - \frac{1}{n} +   (\mbox{Id}_{n \times n}  + L)^k \sum_{w \in W}{a_w \delta_w} \right\|_{L^2}^2 = - \frac{1}{n} +   \left\|   (\mbox{Id}_{n \times n}  + L)^k \sum_{w \in W}{a_w \delta_w} \right\|_{L^2}^2$$
so it suffices to minimize that quantity. This quantity has an explicit interpretation since
$$  (\mbox{Id}_{n \times n}  + L)^k \delta_w $$
is the probability distribution of a random walk started in $w$ after $k$ units of time. This can be summarized as a guiding heuristic:
\begin{quote}
A good graphical design is very nearly characterized by the following property: if we have $a_w$ random walkers start in $w$ for all $w \in W$, then their likelihood of occupying the same vertex at time $t$ is as small as possible.
\end{quote}

Taking a Euclidean manifold and considering a fine grid-like discretization graph invokes probability distributions that
resemble Gaussians. Taking the limit suggests that the behavior in the continuous setting should be given by Brownian
motion and can be rephrased as a packing problem for heat kernels. This is exactly the heuristic derived in \cite{stein}.\\

Moreover, moving away from the notion of graphical design towards, more generally, subsets of weighted vertices that 'almost' integrate the first few eigenfunctions exactly, we
see that 
$$ \left\| - \frac{1}{n} +  (\mbox{Id}_{n \times n}  + L)^k \sum_{w \in W}{a_w \delta_w} \right\|_{L^2}^2 = \sum_{i=2}^{n}{ |\lambda_i + 1|^{2k}  \left| \left\langle \sum_{w \in W}{a_w \delta_w} , \phi_i \right\rangle \right|^2}$$
remains relevant since the right-hand side can be interpreted as a total integration error with a decaying weight on higher oscillations. This identity shows that the
suggested approach based on minimizing random walker interaction remains completely valid even on fairly irregular Graphs that may not display a lot of symmetries or structure.
It remains to be seen whether this can be used for 'numerical integration on graphs' (conceivably useful if the evaluation of $f:V \rightarrow \mathbb{R}$ is costly and it is known that $f$ is smooth w.r.t. the geometry of the graph).

\vspace{-0pt}

\begin{center}
\begin{figure}[h!]
  \begin{tikzpicture}[scale=0.8]
\foreach \a in {1,2,...,30}{
\filldraw (\a*360/30: 3cm) circle (0.09cm);
};
\draw [thick] (360/30: 3cm) --  (2*360/30: 3cm);
\draw [thick] (360/30: 3cm) --  (3*360/30: 3cm);
\draw [thick] (360/30: 3cm) --  (16*360/30: 3cm);
\draw [thick] (360/30: 3cm) --  (24*360/30: 3cm);
\draw [thick] (360/30: 3cm) --  (25*360/30: 3cm);
\draw [thick] (2*360/30: 3cm) --  (4*360/30: 3cm);
\draw [thick] (2*360/30: 3cm) --  (7*360/30: 3cm);
\draw [thick] (2*360/30: 3cm) --  (10*360/30: 3cm);
\draw [thick] (2*360/30: 3cm) --  (30*360/30: 3cm);
\draw [thick] (3*360/30: 3cm) --  (5*360/30: 3cm);
\draw [thick] (3*360/30: 3cm) --  (6*360/30: 3cm);
\draw [thick] (3*360/30: 3cm) --  (11*360/30: 3cm);
\draw [thick] (3*360/30: 3cm) --  (28*360/30: 3cm);
\draw [thick] (4*360/30: 3cm) --  (6*360/30: 3cm);
\draw [thick] (4*360/30: 3cm) --  (13*360/30: 3cm);
\draw [thick] (4*360/30: 3cm) --  (20*360/30: 3cm);
\draw [thick] (4*360/30: 3cm) --  (27*360/30: 3cm);
\draw [thick] (5*360/30: 3cm) --  (7*360/30: 3cm);
\draw [thick] (5*360/30: 3cm) --  (20*360/30: 3cm);
\draw [thick] (5*360/30: 3cm) --  (21*360/30: 3cm);
\draw [thick] (5*360/30: 3cm) --  (29*360/30: 3cm);
\draw [thick] (6*360/30: 3cm) --  (8*360/30: 3cm);
\draw [thick] (6*360/30: 3cm) --  (14*360/30: 3cm);
\draw [thick] (6*360/30: 3cm) --  (23*360/30: 3cm);
\draw [thick] (7*360/30: 3cm) --  (9*360/30: 3cm);
\draw [thick] (7*360/30: 3cm) --  (15*360/30: 3cm);
\draw [thick] (7*360/30: 3cm) --  (26*360/30: 3cm);
\draw [thick] (9*360/30: 3cm) --  (8*360/30: 3cm);
\draw [thick] (10*360/30: 3cm) --  (8*360/30: 3cm);
\draw [thick] (17*360/30: 3cm) --  (8*360/30: 3cm);
\draw [thick] (29*360/30: 3cm) --  (8*360/30: 3cm);
\draw [thick] (9*360/30: 3cm) --  (11*360/30: 3cm);
\draw [thick] (9*360/30: 3cm) --  (22*360/30: 3cm);
\draw [thick] (9*360/30: 3cm) --  (25*360/30: 3cm);
\draw [thick] (10*360/30: 3cm) --  (12*360/30: 3cm);
\draw [thick] (10*360/30: 3cm) --  (18*360/30: 3cm);
\draw [thick] (10*360/30: 3cm) --  (21*360/30: 3cm);
\draw [thick] (11*360/30: 3cm) --  (13*360/30: 3cm);
\draw [thick] (11*360/30: 3cm) --  (18*360/30: 3cm);
\draw [thick] (11*360/30: 3cm) --  (19*360/30: 3cm);
\draw [thick] (12*360/30: 3cm) --  (13*360/30: 3cm);
\draw [thick] (12*360/30: 3cm) --  (14*360/30: 3cm);
\draw [thick] (12*360/30: 3cm) --  (25*360/30: 3cm);
\draw [thick] (12*360/30: 3cm) --  (28*360/30: 3cm);
\draw [thick] (13*360/30: 3cm) --  (15*360/30: 3cm);
\draw [thick] (13*360/30: 3cm) --  (24*360/30: 3cm);
\draw [thick] (14*360/30: 3cm) --  (16*360/30: 3cm);
\draw [thick] (14*360/30: 3cm) --  (19*360/30: 3cm);
\draw [thick] (14*360/30: 3cm) --  (22*360/30: 3cm);
\draw [thick] (15*360/30: 3cm) --  (16*360/30: 3cm);
\draw [thick] (15*360/30: 3cm) --  (17*360/30: 3cm);
\draw [thick] (15*360/30: 3cm) --  (23*360/30: 3cm);
\draw [thick] (16*360/30: 3cm) --  (18*360/30: 3cm);
\draw [thick] (16*360/30: 3cm) --  (29*360/30: 3cm);
\draw [thick] (17*360/30: 3cm) --  (19*360/30: 3cm);
\draw [thick] (17*360/30: 3cm) --  (20*360/30: 3cm);
\draw [thick] (17*360/30: 3cm) --  (28*360/30: 3cm);
\draw [thick] (18*360/30: 3cm) --  (20*360/30: 3cm);
\draw [thick] (18*360/30: 3cm) --  (26*360/30: 3cm);
\draw [thick] (19*360/30: 3cm) --  (21*360/30: 3cm);
\draw [thick] (19*360/30: 3cm) --  (27*360/30: 3cm);
\draw [thick] (20*360/30: 3cm) --  (22*360/30: 3cm);
\draw [thick] (21*360/30: 3cm) --  (23*360/30: 3cm);
\draw [thick] (21*360/30: 3cm) --  (24*360/30: 3cm);
\draw [thick] (22*360/30: 3cm) --  (24*360/30: 3cm);
\draw [thick] (22*360/30: 3cm) --  (30*360/30: 3cm);
\draw [thick] (23*360/30: 3cm) --  (25*360/30: 3cm);
\draw [thick] (23*360/30: 3cm) --  (30*360/30: 3cm);
\draw [thick] (24*360/30: 3cm) --  (26*360/30: 3cm);
\draw [thick] (25*360/30: 3cm) --  (27*360/30: 3cm);
\draw [thick] (26*360/30: 3cm) --  (27*360/30: 3cm);
\draw [thick] (26*360/30: 3cm) --  (28*360/30: 3cm);
\draw [thick] (27*360/30: 3cm) --  (29*360/30: 3cm);
\draw [thick] (28*360/30: 3cm) --  (30*360/30: 3cm);
\draw [thick] (29*360/30: 3cm) --  (30*360/30: 3cm);
\draw [ thick]  (3*360/30: 3cm)  circle  (0.3cm);
\draw [ thick]  (10*360/30: 3cm)  circle  (0.3cm);
\draw [ thick]  (15*360/30: 3cm)  circle  (0.3cm);
\draw [ thick]  (22*360/30: 3cm)  circle  (0.3cm);
\draw [ thick]  (27*360/30: 3cm)  circle  (0.3cm);
   \end{tikzpicture}
\caption{The Wong Graph on 30 vertices: 5 vertices integrate the first 25 eigenfunctions. The $1-$neighborhood is all of $V$.}
\end{figure}
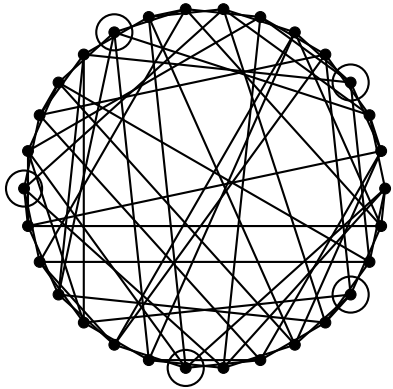
\end{center}

\vspace{-10pt}

\subsection{Open Problems.}
There is a large number of very natural problems. 
We list some.
\begin{enumerate}
\item (Other Laplacians.) Our result works for the Laplace-type operator $L = A D^{-1} - \mbox{Id}_{n \times n}$. There are other natural Laplacians associated to Graphs, see e.g. Chung \cite{chung}. There tends to be
a certain invariance for regular graphs. Regular graphs should therefore be the examples that provide the most 'Laplacian-independent' results.
\item ((Non-)Existence.) Which theoretical results can be proven? How is the existence of a graphical design with certain parameters related to the (a) number of vertices, (b) number of edges, (c) multiplicity of eigenvalues? Given the 'almost-reformulation' as a packing problem even parameters like the independence or chromatic number could play role. 
\item (Weights.) All the examples we constructed in this paper have equal weights $a_w = 1/|W|$. Does adding weights greatly increase the flexibility of graphical designs or is it rather the case that the more extreme cases, for example 4 vertices in the Gosset Graph integrating the first 29 eigenfunctions exactly, are so 'overloaded' with symmetry that allowing for non-constants weights does not yield any improvement at all? It seems natural to assume that in case where weights are effective, their numerical values will encode additional information about their role within the symmetry.

\item (Finding Examples.) Is it possible to completely characterize graphical designs in certain families? There are some families of graphs whose spectra and eigenfunctions are completely characterized (see \cite{bro}),
are some of them suited for a complete analysis? Both the Wang graph (Figure 10) and the Meringer Graph (30 vertices with a graphical design of 6 vertices integrating 25 eigenfunctions exactly) have graphical designs with astonishing degrees of efficiency; they are also both $(5,5)-$cages. The Petersen Graph is also a $(5,5)-$cage and does not seem to have any good (equal-weight) graphical designs. Are there any natural families always supporting high-quality graphical designs?

\item (Fast Computation.) Given an explicit Graph, how would one go about finding a good graphical design or deciding that none with a certain quality exists?

\item (Universal Existence of 'minimal' graphical designs.) It is tempting to conjecture that for every connected Graph (or, as a weaker conjecture, 'generically') and every $1 \leq k \leq n$ there is a set $W \subset V$ of cardinality
$|W| = k$ and weights $a_w$ (not necessarily positive) that integrates at least the first $k$ eigenfunctions exactly. It would be interesting to understand whether the weights can be 
assumed to be positive.

\end{enumerate}

\begin{center}
\begin{figure}[h!]
  \begin{tikzpicture}[scale=0.6]
\foreach \a in {1,2,...,12}{
\filldraw (\a*360/12: 3cm) circle (0.09cm);
};
\foreach \a in {1,2,...,12}{
\draw [thick] (\a*360/12: 3cm) --  (\a*360/12 + 360/12: 3cm);
};
\filldraw [thick] (1, 1) circle (0.09cm);
\filldraw [thick] (-1,1) circle (0.09cm);
\filldraw [thick] (1, -1) circle (0.09cm);
\filldraw [thick] (-1,-1) circle (0.09cm);
\draw [thick] (12*360/12: 3cm) -- (-1,-1) -- (1,1) --  (2*360/12: 3cm);
\draw [thick] (1,1) --  (6*360/12: 3cm);
\draw [thick] (1,1) --  (10*360/12: 3cm);
\draw [thick] (1,-1) --  (3*360/12: 3cm);
\draw [thick] (1,-1) --  (7*360/12: 3cm);
\draw [thick] (1,-1) --  (11*360/12: 3cm);
\draw [thick] (-1,-1) --  (4*360/12: 3cm);
\draw [thick] (-1,-1) --  (8*360/12: 3cm);
\draw [thick] (-1,-1) --  (12*360/12: 3cm);
\draw [thick] (-1,1) --  (1*360/12: 3cm);
\draw [thick] (-1,1) --  (5*360/12: 3cm);
\draw [thick] (-1,1) --  (7*360/12: 3cm);
\draw [thick] (-1,1) -- (1,-1);
\filldraw [thick] (-6, -1.8) circle (0.09cm);
\filldraw [thick] (6, -3) circle (0.09cm);
\filldraw [thick] (1,6) circle (0.09cm);
\draw [thick] (-6, -1.8)  to[out=30,in=180] (12*360/12: 3cm);
\filldraw [thick] (-6, -1.8) -- (3*360/12: 3cm);
\draw [thick] (-6, -1.8)-- (6*360/12: 3cm);
\draw [thick] (-6, -1.8)  to[out=330,in=180] (9*360/12: 3cm);
\filldraw [thick] (6, -3) -- (1*360/12: 3cm);
\filldraw [thick] (6, -3) -- (4*360/12: 3cm);
\filldraw [thick] (6, -3) -- (7*360/12: 3cm);
\filldraw [thick] (6, -3) -- (11*360/12: 3cm);
\filldraw [thick] (1, 6) -- (2*360/12: 3cm);
\filldraw [thick] (1, 6) -- (5*360/12: 3cm);
\draw [thick] (1, 6)  to[out=245,in=80] (8*360/12: 3cm);
\draw [thick] (1, 6)  to[out=300,in=100] (11*360/12: 3cm);
\draw [ thick]  (6,-3) circle (0.4cm);
\draw [ thick]  (-6,-1.8) circle (0.4cm);
\draw [ thick]  (1,6) circle (0.4cm);
   \end{tikzpicture}
\caption{The Robertson Graph on 20 vertices: a subset $W$ of 3 vertices integrates the first 10 eigenfunctions, all those satisfying $|\lambda_k - 1| \geq 0.44$, exactly. The Theorem implies that any such configuration requires the 2-neighborhood of $W$ to have at least 8 vertices.}
\end{figure}
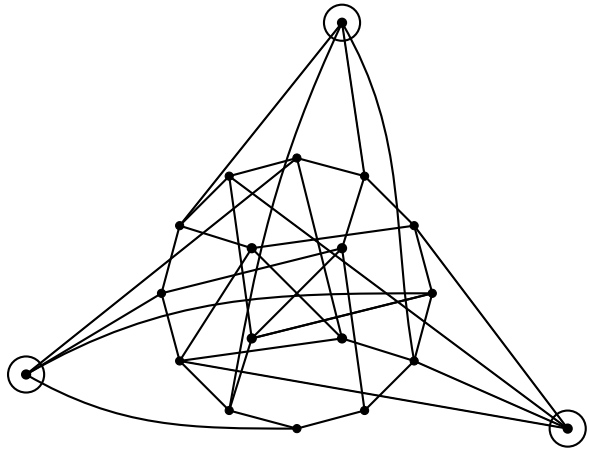
\end{center}

We quickly comment on the last problem and note that a slightly weaker statement is trivial.

\begin{proposition} If $ L = AD^{-1} - \mbox{Id}_{n \times n}$ has an orthogonal set of $n$ eigenvectors, then for every $1 \leq k \leq n$ there exists $W \subset V$ having $|W| = k$ and weights
(not necessarily positive) $a_w$ such that
$$ \sum_{w \in W}{ a_w \phi_{\ell}(w)} = \frac{1}{|V|} \sum_{v \in V}{  \phi_{\ell}(v)}$$
holds for at least $k$ different eigenfunctions (not necessarily the first $k$).
\end{proposition}
\begin{proof} We denote the $n$ eigenvectors by $\phi_1, \dots, \phi_n$ and the $n$ vertices by $v_1, \dots, v_n$. Let us fix $k$ and assume that no such set exists. We construct the matrix
$ A = (\phi_i(v_j))_{i,j=1}^{n}$ and observe that the nonexistence of such sets implies that all $k \times k$ minors have determinant 0. By an iterated Laplace expansion, we can conclude $\det A = 0$ but orthogonality
of the eigenvectors implies $\det A = 1$.
\end{proof}

\end{document}